\documentclass[12pt]{article}

\usepackage{amsthm,amssymb,mathtools} 
\usepackage{color}
\usepackage{enumitem}
\usepackage{hyperref}
\usepackage[subrefformat=parens]{subcaption}
\usepackage{mathrsfs} 

\usepackage{algorithmic}
\usepackage{algorithm}

\usepackage[capitalise]{cleveref}
\crefname{figure}{Figure}{Figures}
\crefname{table}{Table}{Tables}

\usepackage{tikz-cd} 

\usepackage{macros}
\usepackage{mathoperators}
\usepackage{theorems}

\hypersetup{
     colorlinks=true,
     linkcolor=black,
     filecolor=blue,
     citecolor = black,
     urlcolor=cyan,
     }

\def\@themcountersep{}
\usepackage{fancybox,color}
\definecolor{lred}{rgb}{1,0.8,0.5}
\definecolor{lblue}{rgb}{0.8,0.8,1}
\definecolor{dred}{rgb}{0.6,0,0}
\definecolor{dblue}{rgb}{0,0,0.7}
\definecolor{violet}{rgb}{0.5804,0.0000,0.8275}
\definecolor{purple}{rgb}{0.2400,0.5700,0.2500}
\definecolor{TGreen}{rgb}{0,0.50,0.10}

\usepackage[mathlines]{lineno} 
\usepackage{amsmath}

\usepackage{etoolbox} 

\usepackage{easy-todo}

\newcommand*\linenomathpatch[1]{%
    \cspreto{#1}{\linenomath}%
    \cspreto{#1*}{\linenomath}%
    \csappto{end#1}{\endlinenomath}%
    \csappto{end#1*}{\endlinenomath}%
}
\newcommand*\linenomathpatchAMS[1]{%
    \cspreto{#1}{\linenomathAMS}%
    \cspreto{#1*}{\linenomathAMS}%
    \csappto{end#1}{\endlinenomath}%
    \csappto{end#1*}{\endlinenomath}%
}

\expandafter\ifx\linenomath\linenomathWithnumbers
\let\linenomathAMS\linenomathWithnumbers
\patchcmd\linenomathAMS{\advance\postdisplaypenalty\linenopenalty}{}{}{}
\else
\let\linenomathAMS\linenomathNonumbers
\fi

\linenomathpatch{equation}
\linenomathpatchAMS{gather}
\linenomathpatchAMS{multline}
\linenomathpatchAMS{align}
\linenomathpatchAMS{alignat}
\linenomathpatchAMS{flalign}

\title{Exact Matrix Completion via  High-Rank Matrices \\ in Sum-of-Squares Relaxations  } 

\makeatletter
\let\@fnsymbol\@arabic
\makeatother

\author{
\normalsize
    Godai Azuma\thanks{Department of  Mathematical and Computing Science,
        Tokyo Institute of Technology, 2-12-1-W8-29 Oh-Okayama, Meguro-ku, Tokyo 152-8552, Japan.
        ({\tt Makoto.Yamashita@c.titech.ac.jp}).
        The research of Makoto Yamashita was partially supported by JSPS KAKENHI Grant Number JP20H04145.}
        \textsuperscript{,}\thinspace \thanks{Department of Industrial and Systems Engineering,
        Aoyama Gakuin University, 5-10-1-O-410b Fuchinobe, Chuo-ku, Sagamihara-shi, Kanagawa 252-5258, Japan ({\tt azuma@ise.aoyama.ac.jp}).
        The research of Godai Azuma was supported by JSPS KAKENHI Grant Number JP22KJ1307.}
\and
\normalsize
	Sunyoung Kim\thanks{Department of Mathematics, Ewha W. University, 52 Ewhayeodae-gil, Sudaemoon-gu,
	Seoul 	03760, Korea  ({\tt skim@ewha.ac.kr}). This work was supported
        by  NRF 2021-R1A2C1003810.}
\and
\normalsize
        Makoto Yamashita\footnotemark[1]
        }

\begin{document}
\maketitle


\begin{abstract} \noindent
We study  exact matrix completion from partially available data with hidden connectivity patterns.
Exact matrix completion was shown to be possible recently by Cosse and Demanet in 2021
with Lasserre's relaxation using the trace of
the variable matrix as the objective function with given data structured in  a chain format.
In this study, we introduce  a structure for the objective function so that the resulting sum-of-squares (SOS) relaxation, 
the dual of Lasserre's SDP relaxation, produces a  rank-($N$-1) solution, where $N$ denotes the size of variable matrix 
in the  SOS relaxation.
Specifically,  the arrowhead structure is employed for the coefficient matrix of the objective function.
We show that a matrix can be exactly completed  through the SOS relaxation when
the connectivity of  given data   is not explicitly displayed or follows  a chain format.
The theoretical exactness is proved using the rank of the Gram matrix for the SOS  relaxation.
We also present numerical algorithms designed to find the coefficient matrix in the SOS relaxation.
Numerical experiments illustrate the validity of the proposed algorithm.

\end{abstract}

\vspace{0.5cm}


\vspace{0.5cm}


\section{Introduction} \label{sec:introduction}

Matrix completion is to find or estimate the entries of a matrix from partially given entries \cite{Candes2009,Candes2010,Gross2011}.
It has been a  very widely studied subject in recent years for its wide-range applications such as
compressed sensing \cite{CAI2010,Candes2006stable,DONOHO2006}, multi-class learning \cite{Argyriou2008}, dimension reduction \cite{Weinberger2006}.
A basic assumption in the study of matrix completion is that the matrix is low-rank or approximately low-rank  \cite{Candes2015,Candes2009,Chen2018,Keshavan2010,Kiraly2015,Sun2016}.

Recently, it was shown in  \cite{Cosse2021} that the rank-1 matrix can be recovered by
 the Lasserre's semidefinite (SDP) 
 relaxation under some assumptions on the given data. 
The problem of finding the rank-1 matrix $X = \x \trans{\y} \in \Real^{n \times m}$ 
with   $\x \in \Real^n$ and $\y \in \Real^m$ for noiseless case is written as
\begin{equation} \label{eq:MC0}
	\begin{array}{rl}
        \find  & X \in \Real^{n \times m} \\
        \subto & \rank(X) = 1 \\
		       & X_{ij} = (X_0)_{ij} \ \ (i,j) \in \Lambda,
    \end{array}
\end{equation}
where  $ \Lambda$ is a subset of $ \{1,\dots,n\} \times \{1, \dots, m\}$ and  $(X_0)_{ij} $ of $X_0 \in \Real^{n \times m}$  for $(i,j) \in \Lambda$ are given entries.

The condition  for the exact recovery in their work \cite{Cosse2021} was presented with  the connected
bipartite graph $\GC$ associated with $X$ in \eqref{eq:MC0} 
(see Lemma~1 in  \cite{Cosse2021}).
By letting $X = \x\trans{\y}$,
they assumed that at least one connected path in $\Lambda$ for each variable in $\x$ and $\y$ from $x_1$ exist.
Here, $X_{ij} = x_i y_j$ for $(i,j) \in \Lambda$  means that $x_i \to y_j$ or $y_j \to x_i$. 
For instance, for $i_{L} \in \{2, \dots, n\}$, 
a chain for $x_{i_L}$ can be expressed as
\begin{equation*}
  x_1 \to y_{j_1} \to x_{i_1} \to y_{j_2} \to \dots \to 
  y_{j_{L-1}} \to x_{i_{L-1}} \to y_{j_{L}} \to x_{i_{L}}
\end{equation*}
with $\{(1,j_1), (i_1, j_1), (i_1, j_2), \dots, (i_{L-1}, j_{L-1}), 
(i_{L-1}, j_{L}), (i_{L}, j_{L})\} \subset \Lambda$.
For each variable $y_{j_{\bar{L}}} \in \{1, \dots, m\}$, we also have a chain 
that passes through or arrives at $y_{j_{\bar{L}}}$. 
By expressing the variables $\x$ and $\y$ as one vector
 $\z \coloneqq [ x_2, x_3, \ldots, x_n, y_1, y_2, \ldots, y_m]$,
 the chain can be written with $\z$ as
 \begin{align}
  z_0 \to z_{i_1} \to z_{i_2} \to z_{i_3} \to \dots \to z_{i_{L-1}} \to z_{i_L},
  \label{eq:Chain}
 \end{align}
 where $z_0 \coloneqq x_1 = 1$. The vector $\z$ includes the variables from $x_2 (=z_1)$,
 thus, the length of $\z$ 
(the number of variables in \eqref{eq:MC0})
is $s \coloneqq n+m-1$.

In their work \cite{Cosse2021},  the rank of the Gram matrix for the sum of squares (SOS) relaxation \cite{Parrilo2003}, 
the dual of  the Lasserre's SDP relaxation \cite{Lasserre2001}, was shown to be $N-1$ using the condition  \eqref{eq:Chain},
 where  $N$ is the size of the  matrix. As a result, the rank-1 solution was obtained by the Lasserre's SDP relaxation.
 We should mention that  their SDP relaxation for    \eqref{eq:MC0} 
  minimized the trace of the variable matrix, and
 each constraint $ X_{ij} = (X_0)_{ij} \ \ (i,j) \in \Lambda$ of \eqref{eq:MC0} involves only one element of $X$.
For instance,
we consider the constraints with $\x \in \Real^3, \y \in \Real^3$ and $X_{ij} = x_iy_j$: 
\begin{equation} \label{eq:Ex0}
	\begin{alignedat}{8}
		h_1(\z) &= y_1 - (X_0)_{11}    & &= 0, \  & h_2(\z) &= x_2y_1 - (X_0)_{21} & & = 0, \  & h_3(\z) &= x_2y_2 - (X_0)_{22} = 0, \\
		h_4(\z) &= x_3y_1 - (X_0)_{31} & &= 0, \  & h_5(\z) &= x_3y_3 - (X_0)_{33} & & = 0.
	\end{alignedat}
\end{equation}
The chain structure   corresponding to \eqref{eq:Ex0} can be represented as
\[
\begin{array}{ccccccc}
 x_1 &  \longrightarrow &  y_1 & \longrightarrow& x_2  & \longrightarrow & y_2 \\
 & & &  \searrow  & & & \\
  &  & &  & x_3 & \longrightarrow & y_3.\\
\end{array}  
\]
If data is given as a linear combination of $h_k(\z) = 0$ $(k=1,\ldots,5)$, for example,
\begin{alignat*}{2}
	\overline{h}_1(\z) & = 6y_1-8x_2y_1 + 8x_2y_2 + 9x_3y_1 - 10 x_3y_3 & & -\overline{b}_1 = 0, \\
	\overline{h}_2(\z) & = 5x_2y_1 + 9x_2y_2 & & - \overline{b}_2 =0,
\end{alignat*}
 for some $\overline{b}_1, \overline{b}_2 \in \Real$, 
then we need to find a chain that runs through  $\z$ from the given data to see if the problem  can be exactly solved.
The underlying connectivity 
 is called the hidden connectivity in this paper.  
Data with the hidden connvectivity may contain more or less information for a chain. For the above instance, 
 $\overline{h}_1(\z)$ and $\overline{h}_2(\z)$ may contain some or all  monomials 
 $ x_1y_1, x_1y_2, x_1 y_3, x_2y_1, $ $x_2y_2, x_2y_3, x_3y_1, x_3y_2, x_3y_3 $
 with $x_1 = 1$. In this case, a chain can be constructed with $x_1y_1,  x_2y_1, x_2y_2, x_3y_1, x_3y_3$.
 The method in \cite{Cosse2021} for constructing the SOS relaxation  separately dealt with each monomial depending on their
 degree since each constraint in  \eqref{eq:MC0} involves only one variable $X_{ij}$. 
Their method cannot deal with data expressed as a
linear combination of  $h_k(\z) = 0$.


In this paper, we first prove   exact matrix completion for the given data with a chain structure \eqref{eq:Chain} by an SOS relaxation,
 the  dual of the Lasserre's SDP relaxation.
Then, we  show that
the matrix completion problem with the constraints expressed as
 linear combinations of $X_{ij} - (X_0)_{ij} \ (i,j) \in \Lambda$ in \eqref{eq:MC0} can be exactly solved by the SOS relaxation.
 Our approach is  to solve a minimization problem by  employing the  objective function with the arrowhead sparsity pattern~\cite{OLeary1990}
 and converting the equality constraints to inequality constraints.
 We demonstrate that  the objective function  for the minimization problem 
 can be constructed such that  it includes the given data 
and leads to the exact optimal solution.
We also present an algorithm for constructing the objective function with the given data.

\subsection{Main contributions} 
Our main contributions are:
\begin{itemize}

\item While the problem \eqref{eq:MC0} has been studied as an optimization problem with various formulations, 
 there has been limited investigation into the sparsity pattern on the objective function, which may be crucial for  the exact recovery of the matrix.
The arrowhead sparsity structure on the objective function is employed  to achieve the exact recovery of the matrix.

\item For  data with hidden connectivity, 
we show that our proposed method can successfully complete the matrix. 
This is attained through converting  equality constraints into  polynomial inequality constraints  of degree up to 4, which effectively
 reduces the degree of the dual variables in the SOS relaxation.

\item   We also introduce a two-stage algorithm.
The initial stage involves the determination of the objective function, implicitly incorporating the data provided by the constraints, utilizing the arrowhead structured elementary matrix. Subsequently, in the second stage, we solve the SOS relaxation with the objective function  obtained in the first stage and obtain an exact solution.

\end{itemize}

\subsection{Outline of the paper}

Our paper is organized as follows: In Section 2, we introduce some notation and symbols used in this paper and describe the matrix completion problem with an example.
The problem formulation in \cite{Cosse2021} is also described with  its SOS relaxation.
In Section 3, we present our problem formulation using the method in \cite{Kojima05,Kojima03}. In Section 4,  we show theoretical results on the exact matrix recovery 
for data with a chain structure and hidden connectivity. In particular,
we construct the objective function with the  arrowhead sparsity structure  in Section 4.1.
The framework of the proposed algorithm is described in Section 5.1 and numerical results are presented in Section 5.2.
We conclude in Section 6.

\section{Preliminaries} \label{sec:preliminaries}

\subsection{Notation and symbols}

Let $\Natural^s$ and $\Real^s$ be the sets of the $s$-dimensional nonnegative integers and real numbers, respectively.
For  $\boldsymbol{\alpha} =(\alpha_1, \alpha_2, \ldots, \alpha_s) \in \Natural^s$,
a monomial of $\z$ is denoted as $\z^{\boldsymbol{\alpha}} := \prod_{i=1}^s z_i^{\alpha_i}$
and its degree 
$\sum_{i=1}^s \alpha_i$.
A  vector  of all monomials $\z^{\boldsymbol{\alpha}}$ of degree $d$ or less 
 is described as 
\begin{equation*}
    \u_d(\z) \coloneqq \left[\z^{\boldsymbol{\alpha}}\right]_{\sum_{i=1}^s \alpha_i \leq d}
    = \begin{bmatrix} 1 & z_1 & \cdots & z_{s} & z_1^2 & z_1z_2 & \dots & z_{s}^d \end{bmatrix} \in \Real^{s_d},
\end{equation*}
where $s_d ={s+d  \choose d}$. 
We frequently write  $\u_d(\z)$ as $\u_d$, 
in particular, 
\begin{align}
  \u_1 & = [1, \z] := [ 1, x_2, x_3, \dots, x_n, y_1, y_2, \dots, y_m] \in \Real^{s+1}, \nonumber \\
  \u_2 & = [1, x_2, \ldots, x_n,y_1,\ldots, y_m, x_2x_3, x_2x_4, \ldots, y_{m-1}y_m, y_m^2] \in \Real^N, \label{eq:u2}
\end{align} 
  where $N  = {n+m+1  \choose 2}$.

The set of all polynomials in variable vector $\z$ is denoted by $\Real[\z]$, and the set of all polynomial of
degree at most $d$ by $\Real_d[\z]$,
the set of all \textit{nonnegative} polynomials in $\Real[\z]$ by $\NC[\z]$.
A polynomial $p(\z) \in \NC[\z] $ is an SOS polynomial if and only if there exist $r$ polynomials $q_1(\z),\ldots, q_r(\z) \in \Real[\z]$ such that
$p(\z) = \sum_{j=1}^r q_j(\z)^2$.
The set of SOS polynomials of degree $2d$ in $\NC[\z]$ is denoted by $\Sigma_{2d}[\z]$
and its precise definition is given by
\[ \Sigma_{2d}[\z] := \left\{ \sum_{j=1}^r q_j(\z)^2 : r \geq 1, q_j(\z) \in \Real_d[\z]\right\}.
 \]
Obviously, $\Sigma_{2d}[\z] \subseteq \NC[\z] \subseteq \Real[\z]$ holds. 
 An SOS relaxation by SOS polynomials in  $\Sigma_{2d}[\z]$  is denoted by SOS$_{2d}$.
We also use $\Sigma[\z]$ defined by
\[ \Sigma[\z] := \left\{ \sum_{j=1}^r q_j(\z)^2 : r \geq 1, q_j(\z) \in \Real[\z] \right\}.  \]
\begin{lemma} \cite{Choi95} \label{lemmaSOS}
 A polynomial $p(\z) \in \Real_{2d}[\z]$ is an SOS polynomial in $\Sigma_{2d}[\z]$ if and only if there exists
 $W \in \SymMat_+^{s_d}$ such that
\[ p(\z) = \trans{\u_d(\z)} W \u_d(\z), \] 
\end{lemma}
\noindent where $ \SymMat^{s_d}$ and $ \SymMat^{s_d}_+$ denote the sets of symmetric and positive semidefinite matrices of size $s_d$, respectively.
Such a matrix $W$ is called the Gram matrix of $p(\z)$.

A positive semidefinite matrix $R$ is also denoted by $R \succeq O$ or $R \in \SymMat_+$. 
A polynomial matrix is a matrix whose elements are polynomials in $\Real[\z]$.
Let $\Real^{n \times m}[\z]$ denote the set of all $n \times m$ polynomial matrices.
For $M \in \SymMat^{s_d}$,
we use $M_{I}$ to denote the submatrix of $M$ constructed by
the rows and columns of $M$ indexed by $I \subset \{1, \ldots, s_d\}$.

The graph $\GC = (V, E)$ denotes an undirected graph with the vertex set $V$ and the edge set $E$.
In paticular, a bipartite graph is a graph $\GC = (V, E)$ in which
$V$ can be divided to two subsets $L, R \subseteq V$ such that
$L \cap R = \emptyset$ and every edge connects a vertex in $L$ and one in $R$.
It is also be expressed by $\GC = (L, R, E)$.

\subsection{The rank-1 matrix completion problem}
\label{sec:matrix-completion}

To describe the  rank-1 matrix completion problem 
with partial data, 
 an illustrative example is presented:
\begin{example} \label{ex:partial_matrix}
    \begin{equation} \label{eq:example_partial_matrix}
        X_\mathrm{part} =
        \begin{bmatrix}
            7 & 3 & ? \\
            -35 & ? & 10 \\
            ? & 9 & ?
        \end{bmatrix}
    \end{equation}
     The goal of finding the  rank-1 matrix completion here is
    to find  $\x\in \Real^n$ and $\y  \in \Real^m$ from \eqref{eq:example_partial_matrix}, where $X = \x\trans{\y} \in \Real^{n \times m}$ with
    $x_1 = 1$.
    
   If  given data has a unique solution,
    any elements in a partial matrix   $X_\mathrm{part}$ can be computed.
    In the above example, we notice that the second row  of  $X_\mathrm{part}$ is obtained by multiplying the first row by $-5$, thus, 
    we see that the $(2,2)$nd element of $X_\mathrm{part}$ must be $-15$. 
    
     We can formulate the matrix completion of \eqref{eq:example_partial_matrix}
     as: 
    \begin{align*}
    	\begin{array}{rl}
            \find & \z \coloneqq \trans{[x_2, \ldots, x_n, y_1, \ldots, y_m]} \\
            \subto & y_1 = 7, \; y_2 = 3, \; x_2 y_1 = -35, \; x_2 y_3 = 10, \; x_3 y_2 = 9.
        \end{array}
    \end{align*}
\end{example}

Let $\x_0 \in \Real^{n}$ and $\y_0 \in \Real^m$ denote unknown true values of the rank-1 decomposition $X_0 = \x_0 \trans{\y_0}$ 
of $X$, and $ \z_0 \coloneqq \trans{[ (x_0)_2, \dots, (x_0)_n, (y_0)_1, \dots, (y_0)_m]} \in \Real^{s}$. 
Obviously,
if $\z_0$ contains a zero element, then either one of the rows or columns consists entirely of zeros,
and some variables in $z_i$ may not be connected to $x_1$ 
through a chain in $\Lambda$. 
When chains are disconnected, we can decompose the problem into subproblems corresponding to the disconnected components.
We can then apply the  discussion in this paper to each of these components.
Thus, the following is assumed 
without loss of generality. 
\begin{assum} \label{asm:nozeros_in_z}
    There are no zeros in the true value vectors $\x_0$ and $\y_0$, i.e., $\z_0$.
\end{assum}

To express the matrix completion problem in general, we define a function $h_k(\z) \coloneqq x_{i} y_{j} - (\x_0)_i(\y_0)_j$
where $i \in \{2,\ldots,n\}, y \in \{ 1,\ldots, m\}, k = 1, \ldots, K$ and $K$ is the number of partially given data, 
i.e., the cardinality of $\Lambda$.
Then, the matrix completion problem can be written as
\begin{equation}
    \label{eq:matrix_completion_first} \tag{$MC$}
	\begin{array}{rl}
        \find & \z  \\ 
        \subto & h_k(\z) = 0 \quad (k = 1,\ldots,K).
    \end{array}
\end{equation}
The problem can also be  described as a minimization problem with some continuous objective function $h_0(\z)$, for instance,
 $h_0(\z) = \sum_{i=2}^n x_i^2 + \sum_{j=1}^m y_j ^2$:
\begin{equation}
    \label{eq:matrix_completion_1} \tag{$MC^1$}
	\begin{array}{rl}
        \min & h_0 (\z)  \\ 
        \subto & h_k(\z) = 0 \quad (k = 1,\ldots,K).
    \end{array}
\end{equation}
Let $F$  be the feasible region of  \eqref{eq:matrix_completion_1}:
\[ F = \{ \z \in \Real^{s} : h_k(\z) =0 \ (k=1,\ldots, K) \}. \] 
We assume that $F$ is nonempty and bounded. Then,  \eqref{eq:matrix_completion_1} has a finite optimal value $\zeta^*$ at an optimal solution $\z^* \in F$.

 It was shown in \cite{Cosse2021} that problem \eqref{eq:matrix_completion_first} can be solved exactly by the Lasserre's SDP relaxation.
 They proved in \cite[Proposition 3]{Cosse2021}  that for the unique recovery of $X_0$, it is sufficient,
 in addition to the injectivity of the constraints, to find an SOS polynomial $\sum_{j=1}^r q_j(\z)^2 \in \Sigma_4[\z]$ for some $r\geq 1$ with the rank of the associated Gram matrix $N-1$, polynomials  $\lambda_k(\z) \in \Real_{2d-d_k}[\z]$ with $d=2$,   
and  $\rho \in \Real $ such that
\[ q(\z) = \sum_{\scriptsize{\bold{\alpha}} \in \Natural^N_{2}}  \z^{2\scriptsize{\bold{\alpha}}} - \rho + \sum_{k=1}^K \lambda_k(\z) h_k(\z) \ \mbox{ with } q(\z_0) = 0, \] 
where  $\Natural^N_{2} := \left\{\bold{\alpha} \in \Natural^N \mid 
 1 \le \sum_{i=1}^N \alpha_i \leq 2 \right\}$.   
Since $\sum_{\scriptsize{\bold{\alpha}} \in \Natural^N_{d} } \z^{2\scriptsize{\bold{\alpha}}} = \trans{(\u_d)}  \u_d,$
the SOS$_{2d}$ relaxation  in \cite{Cosse2021}, with $d=2$, can be written as
\begin{align}
    \label{eq:cosse_sos_relaxation}
	\begin{array}{rl}
        \max\limits_{\rho, \boldsymbol{\lambda}} & \rho \\
        \subto
            & q(\z) \coloneqq \trans{(\u_d)}  \u_d - \rho + \sum\limits_{k = 1}^K h_k(\z) \lambda_k(\z) \in \Sigma_{2d}[\z]  \\
            & \lambda_k(\z) \in \Real_{2d - d_k}[\z] \quad (k = 1, \ldots, K).
    \end{array}
\end{align}

\section{Problem formulation} \label{sec:main}

We describe
our problem formulation for \eqref{eq:matrix_completion_first}  in this section. The primary goal of our formulation is to handle
the given data with hidden connectivity  for the exact matrix recovery and to efficiently compute $\lambda(\z)$.
We employ the objective function $Q \in \SymMat^N$ with the arrowhead structured  matrix~\cite{OLeary1990} so that
the rank of the Gram matrix for $q(\z)$ becomes $N-1$. 
Moreover, we replace $h_k(\z) = 0 \ (k = 1, \ldots, K)$ in \eqref{eq:matrix_completion_first} with $ h_k(\z)^2 \leq 0$. 

\subsection{Matrix completion problem with connectivity}

Since $h_k(\z) = 0$ is equivalent to 
an inequality $h_k(\z)^2 \leq 0$
for any $\z \in \Real^{s}$ and $k = 1,\dots, K$, 
the matrix completion \eqref{eq:matrix_completion_1} can be equivalently expressed as
\begin{align}
	 \label{eq:matrix_completion} \tag{$MC^2$}
	\begin{array}{rl}
        \min & \  \ h_0(\z)   \\ 
        \subto & -h_k(\z)^2 \geq 0 \quad (k = 1,\ldots,K).
    \end{array}
\end{align}
We define the Lagrangian function $\LC(\z,\bold{\lambda}) : \Real^{s} \times (\Sigma[\z])^K \rightarrow \Real$ for  \eqref{eq:matrix_completion}
\begin{align}
\LC(\z,\bold{\lambda}) = h_0(\z) + \sum_{k=1}^K h_k(\z)^2  \lambda_k(\z). 
\end{align}
Since $F$ corresponds to the feasible region in \eqref{eq:matrix_completion},
it holds that for any $\bold{\lambda} \in (\Sigma[\z])^K$,
\begin{equation*}
	\LC(\z,\bold{\lambda}) \leq h_0(\z) \ \mbox{ if } \z \in F. 
\end{equation*}
We define a Lagrangian dual of  \eqref{eq:matrix_completion}:
\begin{align} \label{eq:LagDual}
\max \; \LC^*(\bold{\lambda}) \quad \subto \;\; \bold{\lambda} \in (\Sigma[\z])^K,
\end{align}
where 
\[ \LC^*(\bold{\lambda}) = \inf \left\{ \LC(\z, \bold{\lambda}) : \z \in \Real^s \right\}. \]
Let $\eta $ be the maximum of the degree of $h_k(\z), \ k=0, 1, \ldots, K$ in  \eqref{eq:matrix_completion}, and
we take $d \in \Natural$ such that $d \ge \lceil \eta/2 \rceil$.
We also let  $2d_k$ be  the degree of  $h_k(\z)^2$. 
Then, we consider a subproblem of the Lagrangian dual \eqref{eq:LagDual}:
\begin{equation} \label{eq:LagDualSub}
	\max \; \LC^*(\boldsymbol{\lambda}) \quad \subto \;\; \lambda_k(\z) \in  \Sigma_{2d-2d_k} \ (k=1, \ldots, K). 
\end{equation}
We can rewrite the problem  \eqref{eq:LagDualSub} as
\begin{equation*}
	\begin{array}{rl}
        \max\limits_{\rho, \boldsymbol{\lambda}} & \rho \\
        \subto
        & \LC(\z,\bold{\lambda}) - \rho \geq 0 \mbox{ \ for every } \z \in \Real^{s}, \\
        & \lambda_k(\z) \in  \Sigma_{2d-2d_k} \quad (k=1,\ldots,K). \rule{0pt}{2.5ex}
    \end{array}
\end{equation*}
By replacing the inequality constraints $ \LC(\z,\bold{\lambda}) - \rho \geq 0 \ (\forall \z \in \Real^s)$ with an SOS condition
$ \LC(\z,\bold{\lambda}) - \rho \in \Sigma_{2d},$ we obtain an SOS relaxation of \eqref{eq:matrix_completion}:
\begin{equation*}
	\begin{array}{rl}
        \max\limits_{\rho, \boldsymbol{\lambda}} & \rho \\
        \subto
        & \LC(\z,\bold{\lambda}) - \rho \in \Sigma_{2d} \mbox{ \ for every } \z \in \Real^{s}, \\
        & \lambda_k(\z) \in  \Sigma_{2d-2d_k} \quad (k=1,\ldots,K). \rule{0pt}{2.5ex}
    \end{array}
\end{equation*}
For $Q \in \SymMat^N$, we let $h_0(\z) = \trans{(\u_2)} Q \u_2$ where $\u_2 $ is given by \eqref{eq:u2}. 
Then the resulting SOS relaxation of  \eqref{eq:matrix_completion} with 
$d=2$ is 
\begin{align}
    \label{eq:sos_relaxation}
	\begin{array}{rl}
        \max\limits_{\rho, \boldsymbol{\lambda}} & \rho \\[-1ex]
        \subto
        & q(\z) := \trans{(\u_2)} Q \u_2 - \rho
            + \sum\limits_{k = 1}^K h_k(\z)^2 \lambda_k(\z)  \in \Sigma_4[\z], \\
        & \lambda_k(\z) \in  \Sigma_{4 - 2d_k}[\z] \quad (k = 1, \ldots, K).
    \end{array}
\end{align}

Reducing  \eqref{eq:sos_relaxation} to a semidefinite program is based on Lemma \ref{lemmaSOS}.
By applying  Lemma \ref{lemmaSOS}
to $q(\z)$ and $h_k(\z)^2 \lambda_k(\z)$, we obtain a semidefinite program for \eqref{eq:sos_relaxation}.
See \cite{Kojima05} for details on the reduction.

We derive Lasserre's SDP relaxation \cite{Lasserre2001} of  \eqref{eq:matrix_completion}  based on the interpretation in \cite{Kojima03}.
Let  $\nu_k = d - d_k \ (k=1,\ldots,K)$.  
The following polynomial optimization problem on positive semidefinite cones (a nonlinear semidefinite program) is equivalent to \eqref{eq:matrix_completion}:
\begin{equation} \label{eq:general_sdp_relaxation}
	\begin{array}{rl}
        \min & \trans{\u_d(\z)} Q \u_d(\z) \\[1ex]
        \subto
        & -h_k(\z)^2 \u_{\nu_k}(\z) \trans{\u_{\nu_k}(\z)} \succeq O \ (k=1,\ldots,K), \\[1ex]
        & \u_d(\z) \trans{\u_d(\z)}  \succeq O.
    \end{array}
\end{equation}
By applying the linearization in \cite{Kojima03} to \eqref{eq:general_sdp_relaxation}, we obtain a linear semidefinite program, which is the dual of 
the semidefinite program obtained from  \eqref{eq:sos_relaxation}.
If $d$ is taken to be $2$, then \eqref{eq:general_sdp_relaxation} corresponds to the Lasserre's SDP relaxation  of \eqref{eq:matrix_completion}.
We call the linear SDP problems obtained through linearizing
\eqref{eq:general_sdp_relaxation}  and \eqref{eq:sos_relaxation} with $d=2$
\textit{the primal SDP relaxation} and \textit{the dual SDP relaxation} of 
\eqref{eq:matrix_completion}, respectively. 

In our formulation \eqref{eq:sos_relaxation}, the objective function $ \trans{\u_d(\z)} Q \u_d(\z)$
 is employed instead of $\trans{\u_d(\z)} \u_d(\z)= \trans{\u_d(\z)} I \u_d(\z)$ in \cite{Cosse2021}, 
 where $I$ denotes the identity matrix. 
The problem~\eqref{eq:sos_relaxation} can be considered as
a generalization of \eqref{eq:cosse_sos_relaxation} 
in the sense
that the coefficient matrix $Q$ of the objective function 
can be chosen to be any symmetric matrix.
We select the structure of 
$Q$ in a manner that ensures the existence of an exact solution. The construction of $Q$ with the arrowhead sparsity pattern
is discussed in \cref{lemma:decomposition_of_El} and \cref{thm:rank_recover}.

\subsection{ Matrix completion problem with hidden connectivity} \label{sec:HiddenConnectivity}

Let $\h(\z) = \trans{[h_1(\z), h_2(\z), \dots, h_K(\z)]}$ be a column vector of 
 the constraints $x_i y_j - (\x_0)_i (\y_0)_j =0$.
We consider a case where  
$h_k(\z)$  includes several terms such as $x_i y_j - (\x_0)_i (\y_0)_j$, $j=1,\ldots,\bar{j}$
for some $\bar{j} >1$, a case more general 
than the previous subsection.
In such a case, the given data may not display the chain format,
instead,
they appear as a  linearly combination of data $x_i y_j - (\x_0)_i (\y_0)_j$, $j=1,\ldots,\bar{j}$, as mentioned in Section~\ref{sec:introduction}. 

A linear combination of $\h(\z)$ can be expressed as 
$\overline{\h}(\z) = C \h(\z)$
with some matrix 
$C \in \Real^{\overline{K} \times K}$. 
We assume that $\trans{C}C$ is invertible for the connectivity of the data.
Then the given data of the problem with hidden connectivity can be represented by $ \overline{h}_{\overline{k}}(\z) = 0 \ (\overline{k}=1,\ldots,\overline{K})$.

As in \eqref{eq:matrix_completion}, we consider 
$ (\overline{h}_{\overline{k}}(\z))^2 \leq 0$ for $\overline{h}_{\overline{k}}(\z) = 0$.
The resulting problem is 
\begin{align}
    \label{eq:matrix_completionH} \tag{$MC^3$}
	\begin{array}{rl}
    \min & h_0(\z)  \\
        \subto &  -( \overline{h}_{\overline{k}}(\z) )^2 \geq 0 \quad (\overline{k} = 1,\ldots,\overline{K}),
    \end{array}
\end{align}
and an SOS relaxation of  \eqref{eq:matrix_completionH} is
\begin{align}
    \label{eq:sos_relaxationH}
	\begin{array}{rl}
        \max\limits_{\rho, \boldsymbol{\lambda}} & \rho \\
        \subto
        & \trans{(\u_2)} Q \u_2 - \rho
            + \sum\limits_{\overline{k} = 1}^{\overline{K}} (\overline{h}_{\overline{k}}(\z))^2 \overline{\lambda}_{\overline{k}}(\z)  \in \Sigma_4[\z], \\
        & \overline{\lambda}_{\overline{k}}(\z) \in \Sigma_{4 - 2d_{\overline{k}}}[\z] \quad (\overline{k} = 1, \ldots, \overline{K}).
    \end{array}
\end{align}

\section{Theoretical results} \label{sec:theory}

In this section,
we construct $Q$ in \eqref{eq:sos_relaxation} in 
\cref{lemma:decomposition_of_El} and \cref{thm:rank_recover} and
  provide  a theoretical proof for the exact recovery with the SOS relaxation~\eqref{eq:sos_relaxation}. 
  In \cref{sec:given-data}, we show that $Q$ constructed in \cref{sec:Qconst} implicitly includes the given data with $\h(\z)$.   
The problem with hidden connectivity is discussed in \cref{sec:hidden}.

\subsection{Exact recovery with  SOS relaxation~\texorpdfstring{\eqref{eq:sos_relaxation}}{}}
 \label{sec:Qconst}
 
We define
the  constraint  of  the  SOS relaxation~\eqref{eq:sos_relaxation}  as
\begin{equation} \label{eq:target_sos}
	q(\z; Q; \rho, \boldsymbol{\lambda}) \coloneqq
	\trans{(\u_2)} Q \u_2 - \rho + \sum\limits_{k = 1}^K h_k(\z)^2 \lambda_k(\z)
    \in \Sigma_4[\z].
\end{equation}
Since $q(\z; Q; \rho, \boldsymbol{\lambda}) \in \Sigma_4[\z]$, there exists a positive semidefinite matrix $\Gamma \in \SymMat^N_+$ such that
 $q(\z; Q; \rho, \boldsymbol{\lambda}) = \trans{(\u_2)} \Gamma (\u_2)$.
If the rank of $\Gamma$ is $N-1$, the positive semidefiniteness of $\Gamma$ gurantees that there exists 
a unique vector $\u_2$ such that $\trans{(\u_2)} \Gamma (\u_2) = 0$ 
and $(\u_2)_1 = 1$.
This implies that
there exists a unique point $\bar{\z} \in \Real^s$ such that $q(\bar{\z}; Q; \rho, \boldsymbol{\lambda} ) = 0$.
When searching for a point $\bar{\z}$ that leads  
$q(\bar{\z}; Q; \rho, \boldsymbol{\lambda}) = 0$,  a natural candidate is a vector consisting of the true data values, denoted by $\z_0$. Let $(\u_2)_0 := \u_2(\z_0)$. 
Since $\h(\z_0) = 0$, 
 $\sum_{k = 1}^K h_k(\z_0)^2 \lambda_k(\z_0) = 0$ in \eqref{eq:target_sos}. Thus, 
\begin{align*}
0 = q(\z_0; Q; \rho, \boldsymbol{\lambda}) = \trans{(\u_2)}_0 Q (\u_2)_0 - \rho = \trans{(\u_2)_0} \Gamma (\u_2)_0,
\end{align*}
which implies
\begin{align} \label{eq:rho}
  \trans{(\u_2)_0} Q (\u_2)_0 - \rho = 0.
 \end{align} 
Therefore, we can take  $\bar{\z} = \z_0$.
 
For the uniqueness of $\bar{\z}$, the rank of $\Gamma$ needs to be $N-1$. 
Thus, $Q$ must be chosen such that the rank of $Q$ is $N-1$. 
Recall that $Q \in \SymMat^N$ can be any symmetric matrix in \eqref{eq:sos_relaxation}. Specifically, 
we utilize the arrowhead sparsity pattern  for the rank of $Q$ to be $N-1$.

To construct  $Q$ whose rank is $N-1$ using \eqref{eq:rho},
we first  consider 
 $\trans{(\u_2)} Q \u_2 - \trans{(\u_2)}_0 Q (\u_2)_0$ in \eqref{eq:target_sos}.
Let $E^r \in \SymMat^{N}$ be a diagonal matrix whose $(r,r)$th element is one and all others are zeros.

\begin{lemma}
    \label{lemma:decomposition_of_El}
    For any $r \in \{2 \ldots, N\}$,
    there exists 
    $Q^r \in \SymMat_+^{N}$ such that
    \begin{equation} \label{eq:decomposition_of_El}
		\trans{(\u_2)} Q^r \u_2 - \trans{(\u_2)}_0 Q^r (\u_2)_0 = \trans{\left(\u_2 - (\u_2)_0\right)} E^r \left(\u_2 - (\u_2)_0\right).
    \end{equation}
\end{lemma}
\begin{proof}
    We define two monomials $z_{i(r)}$ and $z_{j(r)}$ of degree $0$ or $1$ so that the $r$th component of 
    $\u_2$ corresponds to the product of $z_{i(r)}$ and $z_{j(r)}$, i.e., $(\u_2)_r = z_{i(r)} z_{j(r)}$.
    Let $Q^r \in \SymMat_+^{N}$ be a matrix with only four nonzero elements in 
	\begin{align} \label{eq:definition_of_Qell}
       [Q^r]_{11} = ((\z_0)_{i(r)})^2 ((\z_0)_{j(r)})^2, \quad  
      [Q^r]_{1r}  = [Q^r]_{r1} = - (\z_0)_{i(r)} (\z_0)_{j(r)},   \quad
      [Q^r]_{rr}   = 1.
  \end{align}
Then, we have
\begin{align}
    \trans{(\u_2)}_0 Q^r (\u_2)_0 & = 
				((\z_0)_{i(r)})^2 ((\z_0)_{j(r)})^2 - 2(\z_0)_{i(r)}(\z_0)_{j(r)} (\z_0)_{i(r)}(\z_0)_{j(r)} + ((\z_0)_{i(r)})^2 ((\z_0)_{j(r)})^2 \nonumber \\
			& = 0.
            \label{eq:u02Qru02}
\end{align}
	By \eqref{eq:definition_of_Qell} and \eqref{eq:u02Qru02},
	\begin{align*}
		 \trans{(\u_2)} Q^r \u_2 - \trans{(\u_2)}_0 Q^r (\u_2)_0
		= & \  z_{i(r)}^2 z_{j(r)}^2 - 2(\z_0)_{i(r)}(\z_0)_{j(r)} z_{i(r)}z_{j(r)} + ((\z_0)_{i(r)})^2 ((\z_0)_{j(r)})^2 \nonumber \\
		= & \ \trans{\left(\u_2 - (\u_2)_0\right)} E^{r} \left(\u_2 - (\u_2)_0\right).
	\end{align*} 
\end{proof}

Next, we show that for any ${\boldsymbol\lambda}$,
there exists a matrix $Q$ such that
the rank of  the Gram matrix
for $q(\z; Q; \rho, \boldsymbol{\lambda})$ is of $N-1$.
\begin{theorem}
    \label{thm:rank_recover}
	Suppose that Assumption~\ref{asm:nozeros_in_z} holds.
	Then, there exist	
	$Q \in \SymMat^N_+,  \overline{E}\in \SymMat^N_+, \Gamma \in \SymMat^N_+$
 such that
	\begin{equation} \label{eq:equality_of_Q_and_R}
		\trans{(\u_2)} Q \u_2 - \trans{(\u_2)}_0 Q (\u_2)_0 = \trans{\left(\u_2 - (\u_2)_0\right)} \overline{E} \left(\u_2 - (\u_2)_0\right)
		= \trans{(\u_2)} \Gamma  \u_2, 
	\end{equation}
	and the ranks of $Q, \overline{E}$ and $\Gamma$ are all $N-1$.
	As a result,  $\trans{(\u_2)} Q \u_2 - \trans{(\u_2)}_0 Q (\u_2)_0$ is an SOS polynomial.
     In particular,  $Q$ can be constructed with the arrowhead sparsity pattern.
      
\end{theorem}
\begin{proof}
    By \cref{lemma:decomposition_of_El}, there exists $Q^r \in \SymMat_+^n$ satisfying \eqref{eq:decomposition_of_El} for $r \in \{2, \dots, N\}$.
    We can obtain the first equality of \eqref{eq:equality_of_Q_and_R}
	by letting $Q = \sum_{r = 2}^{N} Q^r$ and $\overline{E} = \sum_{r = 2}^{N} E^r$.
	Obviously, $\rank \overline{E} = N - 1$.
	Next, we check the rank of the matrix $Q$.
    Since each $Q^r$ in \cref{lemma:decomposition_of_El} is a rank-1 matrix, the subadditivity of the rank implies $\rank Q \leq N - 1$.
	By \eqref{eq:decomposition_of_El},
	the submatrix $Q_{\{2, \ldots, N\}}$ is a diagonal matrix with ones on the diagonal. 
	Hence, it follows $\rank Q = N - 1$.
    Furthermore, we can take $\Gamma = Q$ by \eqref{eq:u02Qru02}, 
    and this implies that $\rank\Gamma = N-1$.
\end{proof}
It was shown in \cite{Azuma2021} that the exact solution of quadratic programs with quadratic constraints (QCQPs) can be obtained when the
aggregated sparsity pattern of the QCQP forms a tree. We mention that the arrowhead sparsity pattern forms a tree.

The following remark 
is a consequence of \cref{thm:rank_recover}
for $\rho = \trans{(\u_2)}_0 Q (\u_2)_0$ and $\boldsymbol{\lambda} = \0$.
\begin{rema} \label{rem:existence_Q_for_exact} 
    Assume that
   $\sum_{k=1}^K h_k(\z)^2 \lambda_k(\z)$ 
   with $\boldsymbol{\lambda} = \0$
    are added to the left hand side of \eqref{eq:equality_of_Q_and_R}. 
    Then, we have
    \begin{equation} \label{eq:equality_of_Q_and_R1}
		\trans{(\u_2)} Q \u_2 - \trans{(\u_2)}_0 Q (\u_2)_0 + \sum\limits_{k=1}^K h_k(\z)^2 \lambda_k(\z) = \trans{\left(\u_2 - (\u_2)_0\right)} \overline{E} \left(\u_2 - (\u_2)_0\right).
	\end{equation}	
    Then, it can be viewed that the results in \cref{thm:rank_recover} holds 
    with the choice of
     $\boldsymbol{\lambda} = \0$ in \eqref{eq:equality_of_Q_and_R1}.
    In other words, 
    when considering \eqref{eq:matrix_completion} for any rank-$1$ matrix completion problem,
    there always exists  $Q \in \SymMat_+^{N}$ of rank exactly $N - 1$ guaranteeing
    the existence of $\rho \in \bR$ and $\lambda_k(\z) \in \Sigma_{4 - 2d_k}[\z]$ such that
    $q(\z; Q; \rho, \boldsymbol{\lambda} )  \in \Sigma_4[\z]$.
     Moreover, $q(\z; Q; \rho, \boldsymbol{\lambda})$ can be represented with a rank-($N - 1$) Gram matrix.
\end{rema}

For nonzero $\lambda_k(\z)$, we show that
$h_k(\z)^2 \lambda_k(\z)$ does not decrease
the rank of the Gram matrix of \eqref{eq:target_sos}.
\begin{theorem} \label{thm:rank_recover_gen}
	Let $Q^r$ be an $N \times N$ symmetric positive semidefinite matrix defined by \eqref{eq:definition_of_Qell},
	and $Q \coloneqq \sum_{r = 2}^N Q^r$.
	For any $\rho \in \Real$ and $\lambda_k(\z) \in \Sigma_{4-2d_k}$ ($k = 1,\ldots,K$)
	such that $q(\z; Q; \rho, \boldsymbol{\lambda})$ is an SOS polynomial,
	the rank of the Gram matrix $\Gamma \in \SymMat^N_+$ for $q(\z; Q; \rho, \boldsymbol{\lambda})$,
    i.e., $\trans{(\u)_2} \Gamma \u_2 = q(\z; Q; \rho, \boldsymbol{\lambda})$,
	is  greater than or equal to $(N - 1)$.
\end{theorem}
\begin{proof}
	In the submatrix $Q_{\{2,\ldots,N\}}$ that includes the second  to the $N$th rows and columns of $Q$,
	the diagonal elements are positive numbers and all the other elements are zeros by \eqref{eq:definition_of_Qell}.
	Hence, the submatrix $Q_{\{2,\ldots,N\}}$ is a positive definite matrix.
	For any $k = 1,\ldots,K$,
	the polynomial $ h_k^2(\z) \lambda_k(\z) \in \Sigma(\z)$ from $\lambda_k(\z) \in \Sigma_{4-2d_k}(\z)$ and $h_k^2(\z) \in \Sigma_{2d_k}(\z)$,
	and there exists an $N \times N$ symmetric matrix $\Delta \in \SymMat_+^N$ such that
	\begin{equation*}
		\trans{(\u_2)} \Delta \u_2 = \sum_{k = 1}^K h_k^2(\z)  \lambda_k(\z).
	\end{equation*}
	Obviously, $\Delta_{\{2,\ldots,N\}} \succeq O$.
    For  $ p \in \Natural$,
    let $A \in \SymMat^p$ be a positive definite matrix, and let $B$ be a positive semidefinite matrix.
    It is easy to show that $A + B$ is also a positive definite matrix, i.e., the rank of $A + B$ is $p$.
	Thus, $Q_{\{2,\ldots,N\}} + \Delta_{\{2,\ldots,N\}}$ is a positive definite matrix.
	The rank of $Q - \rho E^1 + \Delta$ is greater than or equal to $N-1$ and
	$Q - \rho E^1 + \Delta$ is a Gram matrix for  $q(\z; Q; \rho, \boldsymbol{\lambda})$.
	\end{proof}

\begin{rema}\label{rema:rank-one}
    The rank-$(N-1)$ Gram matrix is essential for the 
    exactness of the primal SDP relaxation  which can be obtained by linearizing  \eqref{eq:general_sdp_relaxation}.

Since the feasible region $F$ of \eqref{eq:matrix_completion}
is nonempty and bounded as assumed in \cref{sec:matrix-completion},
its primal SDP relaxation is feasible. In fact,
$\z \in F$  satisfies $h_k(\z)^2 = 0$, thus, 
$-h_k(\z)^2 \u_{\nu_k}(\z) \trans{\u_{\nu_k}(\z)}$ is the zero matrix.

On the other hand, in \eqref{eq:sos_relaxation}, we can take $\lambda_k(\z) = 1$
for each $k=1, \dots, K$ and a  negative $\rho$ with very large magnitude. 
Then, we can find a positive definite matrix $\overline{W}$ satisfying
$\trans{(\u_2)} Q \u_2 - \rho + \sum\limits_{k = 1}^K h_k(\z)^2 \lambda_k(\z) = \trans{\u_d(\z)} \overline{W} \u_d(\z)$.
Thus, the dual SDP relaxation has an interior feasible point,
and strong duality holds for the primal and dual SDP relaxation.
Therefore, by the discussion in Section 3 of \cite{Azuma2021}, 
when the Gram matrix is of rank-$(N-1)$,
the primal SDP relaxation provides an exact solution to \eqref{eq:matrix_completion_first}.    
More precisely, let $X^*$ and $S(\z^*)$ be the solutions of the primal and dual SDP relaxation. Then,
    $X^*S(\z^*) = O$ by  strong duality.
    From the Sylvester's rank inequality~\cite{Anton2014},
        $\rank{X^*} + \rank{S(\z^*)} \leq N + \rank{X^*S(\z^*)}$ holds for $X^*$ and $S(\z^*)$.
    Therefore, $\rank{X^*} \leq N + \rank{O} - \rank{S(\z^*)} \leq N + 0 - (N - 1) = 1$.
    Consequently, this rank-1 condition ensures the exactness of SDP relaxation \cite{Azuma2021}.
\end{rema}

\subsection{Given partial data and the SOS expression} 
\label{sec:given-data}
\label{sec:expression-with-h}

While the SOS relaxation \eqref{eq:target_sos} includes
$\sum\limits_{k = 1}^K h_k(\z)^2 \lambda_k(\z)$,
\eqref{eq:equality_of_Q_and_R} holds without explicit dependence on $\h(\z)$.
We  show that  $\h(\z)$ is implicitly included in $\trans{\left(\u_2 - (\u_2)_0\right)} \overline{E} \left(\u_2 - (\u_2)_0\right)$
 in \eqref{eq:equality_of_Q_and_R} in this subsection.  

 Suppose that we have an arbitrary chain from the chains expressed as \eqref{eq:Chain}.
 Then,  the indices for the elements of $\h(\z)$ may not coincide with the data indices, for instance,  $h_2(\z) = z_7 z_3 - (\z_0)_7 (\z_0)_3$.  
 In this case,  $\h(\z)$ can be renumbered according to a chain structure
 as follows:
 \begin{align}
  h_{i_1}(\z) = z_{i_1} - (\z_0)_{i_1}, \quad h_{i_{\ell}}(\z) = 
  z_{i_{\ell-1}}  z_{i_{\ell}} - (\z_0)_{i_{\ell-1}}  (\z_0)_{i_{\ell}}, 
 \  \ell = 2, \dots, L. \label{eq:chain-constraints}
 \end{align}
 By the definition of $\u_2$ in \eqref{eq:u2},
we see that $\trans{\left(\u_2 - (\u_2)_0\right)} \overline{E} \left(\u_2 - (\u_2)_0\right)$ 
includes
 the terms of form 
 $ (z_i -(\z_0)_i  )^2 \ (i=1,\ldots,s) $ or $ (z_i z_j-(\z_0)_i (\z_0)_j)^2$ \ $(1 \leq i \le j \leq s)$.
 
First, let us focus on the chain in \eqref{eq:chain-constraints}
and express 
$ (z_{i_\ell} -(\z_0)_{i_\ell} )^2 \ (\ell=1,\ldots,L)$
as $\trans{(\h(\z) \otimes \u_2)} U_{i_{\ell}}  (\h(\z) \otimes \u_2)$
where $\otimes$ is the Kronecker product and 
$U_{i_\ell} \in \SymMat_+^{KN}$.
We see that $h_k(\z) z_i$ is included in the vectors $\h(\z) \otimes \u_1$ and   $\h(\z) \otimes \u_2$, and 
$h_k(\z) z_i z_j$ in   $\h(\z) \otimes \u_2$. 
To simplify the discussion with  $U_{i_\ell}$, 
we use the  notation   $I^1(h_k(\z) z_i)$ and $I^2(h_k(\z) z_i z_j)$ to denote 
the indices that correspond to 
$h_k(\z) z_i$ in  $\h(\z) \otimes \u_1$ and 
$h_k(\z) z_i z_j$ in $\h(\z) \otimes \u_2$,
i.e., 
$(\h(\z) \otimes \u_1)_{I^1(h_k(\z) z_i)} = h_k(\z) z_i$
and $(\h(\z) \otimes \u_2)_{I^2(h_k(\z) z_i z_j)} = h_k(\z) z_i z_j$.
We also use the same notation for $z_i = 1$ or $z_j = 1$.
For example, the first element of $\u_2$ is 1, thus,
$I^2(h_k(\z)) = (k-1)N+1$ and $(\h(\z) \otimes \u_2)_{I^2(h_k(\z))} = h_k(\z)$.
In addition, we also use $\e_i$ to denote a vector whose $i$th component is 1 and all other components are 0.

Next, we examine each term corresponding to a constraint which forms the chain in \eqref{eq:chain-constraints}.
The  term  $(z_{i_1} -(\z_0)_{i_1} )^2$ in 
$\trans{\left(\u_2 - (\u_2)_0\right)} \overline{E} \left(\u_2 - (\u_2)_0\right)$ 
corresponds to the first constraint in the chain
 and it  can be written as
\begin{align}
  z_{i_1} - (\z_0)_{i_1} = h_{i_1}(\z) = \trans{(\e_{I^1(h_{i_1}(\z))})}(\h(\z) \otimes \u_1), \label{eq:a_1}
\end{align}
from \eqref{eq:chain-constraints}.
Therefore, 
\begin{equation*}
  (z_{i_1} -(\z_0)_{i_1} )^2 = \trans{(\h(\z) \otimes \u_1)} U^1_{i_1}  (\h(\z) \otimes \u_1)  
\end{equation*}
holds with $U^1_{i_1} := E^{I^1(h_{i_1}(\z))} \in \SymMat_+^{s+1}$.

For the second term $(z_{i_2} -(\z_0)_{i_2})^2$, we use a relation:
\begin{align}
  z_{i_2} - (\z_0)_{i_2}  
  = & \ \frac{-1}{(\z_0)_{i_1}} (z_{i_1} - (\z_0)_{i_1}) z_{i_2} 
  + \frac{1}{(\z_0)_{i_1}}(z_{i_1} z_{i_2} -  (\z_0)_{i_1} (\z_0)_{i_2}) \nonumber \\
= & \ \frac{-1}{(\z_0)_{i_1}} h_{i_1}(\z) z_{i_2} 
+ \frac{1}{(\z_0)_{i_1}} h_{i_2}(\z) \nonumber \\
= & \trans{\left(\frac{-1}{(\z_0)_{i_1}} \e_{I^1(h_{i_1}(\z) z_{i_2})}
+ \frac{1}{(\z_0)_{i_1}} \e_{I^1(h_{i_2}(\z))} \right)} (\h(\z) \otimes \u_1). 
\label{eq:a_2}  
\end{align}
Therefore, 
\begin{equation*}
  (z_{i_2} -(\z_0)_{i_2} )^2
  =  \trans{(\h(\z) \otimes \u_1)} U^1_{i_2}  (\h(\z) \otimes \u_1)
\end{equation*}
can be obtained with 
\begin{equation*}
  U^1_{i_2} := 
\left(\frac{-1}{(\z_0)_{i_1}} \e_{I^1(h_{i_1}(\z) z_{i_2})}
+ \frac{1}{(\z_0)_{i_1}} \e_{I^1(h_{i_2}(\z))} \right)
\trans{\left(\frac{-1}{(\z_0)_{i_1}} \e_{I^1(h_{i_1}(\z) z_{i_2})}
+ \frac{1}{(\z_0)_{i_1}} \e_{I^1(h_{i_2}(\z))} \right)}.
\end{equation*}

We apply the induction step on $\ell=3,\dots,L$, i.e.,  $(z_{i_{\ell}} - (\z_0)_{i_{\ell}})^2$
is examined.
Assume that there exists
a vector $\a_{i_{\overline{\ell}}} \in \Real^{K(s+1)}$ for each $\overline{\ell} = 1, \dots, \ell-1$ 
 such that 
$z_{i_{\overline{\ell}}} - (z_{i_{\overline{\ell}}}) = \trans{\a_{i_{\overline{\ell}}}} (\h(\z) \otimes \u_1)$.
In fact, \eqref{eq:a_1} and \eqref{eq:a_2}
indicate that we can take  $\a_{i_1} = \e_{I^1(h_{i_1}(\z))}$
and $\a_{i_2} = 
\frac{-1}{(\z_0)_{i_1}} \e_{I^1(h_{i_1}(\z) z_{i_2})}
+ \frac{1}{(\z_0)_{i_1}} \e_{I^1(h_{i_2}(\z))}$, respectively.
As a result, we have 
\begin{align*}
   z_{i_{\ell}} - (\z_0)_{i_\ell}   =  &  \  
   \frac{1}{(\z_0)_{i_{\ell-2}} (\z_0)_{i_{\ell-1}}}
    (z_{i_{\ell-1}} z_{i_{\ell}}  - (\z_0)_{i_{\ell-1}} (\z_0)_{i_{\ell}} ) 
    z_{i_{\ell-2}} \nonumber \\
    & \ - 
    \frac{1}{(\z_0)_{i_{\ell-2}} (\z_0)_{i_{\ell-1}}}
    (z_{i_{\ell-2}} z_{i_{\ell-1}} - (\z_0)_{i_{\ell-2}} (\z_0)_{i_{\ell-1}} ) z_{i_{\ell}} \nonumber \\
    &  \ +
    \frac{(\z_0)_{i_{\ell}} (\z_0)_{i_{\ell-1}}}{(\z_0)_{i_{\ell-2}} (\z_0)_{i_{\ell-1}}}
     (z_{i_{\ell-2}} - (\z_0)_{i_{\ell-2}} ) \nonumber \\
    =  &  \ 
    \trans{\a_{i_\ell}} (\h(\z) \otimes \u_1),
  \end{align*}
  where  $\a_{i_{\ell}}$ is defined by
  \begin{align}
    \a_{i_{\ell}} = \frac{1}{(\z_0)_{i_{\ell-2}} (\z_0)_{i_{\ell-1}}} 
    \e_{I^1(h_{i_\ell}(\z) z_{i_{\ell-2}})} 
    - 
    \frac{1}{(\z_0)_{i_{\ell-2}} (\z_0)_{i_{\ell-1}}}
    \e_{I^1(h_{i_{\ell-1}}(\z) z_{i_{\ell}})} 
    +
    \frac{(\z_0)_{i_{\ell}} (\z_0)_{i_{\ell-1}}}{(\z_0)_{i_{\ell-2}} (\z_0)_{i_{\ell-1}}} \a_{i_{\ell-2}}.
    \label{eq:a_l}
  \end{align}
By letting $U^1_{i_{\ell}} = \a_{i_\ell} \trans{\a_{i_{\ell}}} \in \SymMat_+^{K(s+1)}$,
we have that
\begin{align} \label{eq:U1}
  (z_{i_{\ell}} -(\z_0)_{i_\ell} )^2
  =  \trans{(\h(\z) \otimes \u_1)} U^1_{i_\ell}  (\h(\z) \otimes \u_1)
\end{align}
for $\ell = 3, \dots, L$.
Since the elements of $\u_2$ 
include those of $\u_1$,
$U^1_{i_{\ell}} \in \SymMat_+^{K(s+1)}$ can be expanded into a matrix $U^2_{i_{\ell}} \in \SymMat_+^{KN}$ such that 
$\trans{(\h(\z) \otimes \u_1)} U^1_{i_\ell}  (\h(\z) \otimes \u_1)
= \trans{(\h(\z) \otimes \u_2)} U^2_{i_\ell}  (\h(\z) \otimes \u_2)$
by inserting zeros in $U^2_{i_{\ell}}$ appropriately.

By applying the discussion above to each chain
starting from $z_0 = x_1$,
we can find $U_i^2 \in \SymMat_+^{KN}$ for each $i=1,\ldots,s$ such that 
\begin{align*}
  (z_{i} -(\z_0)_{i} )^2
  =  \trans{(\h(\z) \otimes \u_2)} U_{i}^2  (\h(\z) \otimes \u_2).
\end{align*}

Now, we consider $ (z_i z_j-(\z_0)_i (\z_0)_j)^2$ \ $(1 \leq i < j \leq s)$ since
 $\trans{\left(\u_2 - (\u_2)_0\right)} \overline{E} \left(\u_2 - (\u_2)_0\right)$  includes not only   $  (z_{i} -(\z_0)_{i} )^2$ but also
$ (z_i z_j-(\z_0)_i (\z_0)_j)^2$ \ $(1 \leq i < j \leq s)$.
By the relation \eqref{eq:a_l}, we can find $\hat{\a}_{i}$ and $\hat{\a}_{j}$
such that $(z_i - (\z_0)_i) = \trans{(\hat{\a}_{i})} (\h(\z) \otimes \u_1)$
and $(z_j - (\z_0)_j) = \trans{(\hat{\a}_{j})} (\h(\z) \otimes \u_1)$.
We also notice that $\u_1$ includes $z_j$. 
As a result, we write as follows: 
\begin{align} \label{eq:factor}
  z_i z_j  - (\z_0)_i (\z_0)_{j} 
 = & \  (z_i - (\z_0)_i) z_j 
  + (\z_0)_{i}  (z_j - (\z_0)_j) \nonumber \\
  = & \ \trans{(\hat{\a}_{i})} (\h(\z) \otimes \u_1) (\trans{(\e_{I^1(z_j)})} \u_1)
  + (\z_0)_{i} \trans{(\hat{\a}_{j})} (\h(\z) \otimes \u_1).
\end{align} 
We see that the Kronecker  product $\h(\z) \otimes \u_2$, i.e., the componentwise product of $\h(\z)$ and $\u_2$,
 includes
 $\h(\z) \otimes \u_1$.
Thus,  there exists $U^2_{ij} \in \SymMat_+^{KN}$  such that 
$(z_i z_j  - (\z_0)_i (\z_0)_{j})^2 = \trans{(\h(\z)\otimes \u_2)} U^2_{ij}  (\h(\z)\otimes \u_2)$.

Consequently, it holds that
\begin{align*}
  & \ \trans{\left(\u_2 - (\u_2)_0\right)} \overline{E} \left(\u_2 - (\u_2)_0\right) \\
  = & \ \sum_{i=1}^s (z_i - (\z_0)_i)^2 
  + \sum_{i=1}^{s} \sum_{j=i}^{s} (z_i z_j - (\z_0)_i (\z_0)_j)^2 \\
 = & \ \sum_{i=1}^s \trans{(\h(\z)\otimes \u_2)} U^2_{i}  (\h(\z)\otimes \u_2)
+ \sum_{i=1}^{s} \sum_{j=i}^{s} \trans{(\h(\z)\otimes \u_2)} U^2_{ij}  (\h(\z)\otimes \u_2) \\
= & \ \trans{(\h(\z)\otimes \u_2)} U (\h(\z)\otimes \u_2),
\end{align*}
where 
\begin{equation} \label{eq:U_def}
U := \sum_{i=1}^s  U^2_{i} + \sum_{i=1}^{s} \sum_{j=i}^{s} U^2_{ij}.
\end{equation}
It should be mentioned that different ways of expressing 
\eqref{eq:factor} exist, for example,
\begin{equation*} 
  z_i z_j  - (\z_0)_i (\z_0)_{j} 
 =  \  (z_j - (\z_0)_j) z_i
  + (\z_0)_{j}  (z_i - (\z_0)_i), 
\end{equation*}
therefore, the matrix $U$ is not always unique.

In view of \cref{rem:existence_Q_for_exact}, 
we can find $Q \in \SymMat_+^{N}$, $\lambda_k(\z)  \in \Sigma_{4-2d_k}{[\z]} \ (k=1,\dots,K)$ and $U \in \SymMat_+^{KN}$ 
such that 
\begin{equation} \label{eq:equality_of_Q_and_R2}
  \trans{(\u_2)} Q \u_2 - \trans{(\u_2)}_0 Q (\u_2)_0 + \sum\limits_{k=1}^K h_k(\z)^2 \lambda_k(\z) = \trans{(\h(\z)\otimes \u_2)} U (\h(\z)\otimes \u_2).
\end{equation}	
Consequently,  we have shown that $\h(\z)$ is included in $\trans{\left(\u_2 - (\u_2)_0\right)} \overline{E} \left(\u_2 - (\u_2)_0\right)$
 in \eqref{eq:equality_of_Q_and_R}.

\subsection{Exact recovery 
of the problem with  hidden  connectivity} \label{sec:hidden}

In this subsection, we discuss 
the exact recovery of problems  
with hidden connectivity in \cref{sec:HiddenConnectivity}.

Using $\overline{\bold{h}}(\z) = C \h(\z)$ in \eqref{eq:sos_relaxationH},
the equation \eqref{eq:equality_of_Q_and_R2} turns  out  to be
\begin{equation} \label{eq:equality_of_Q_and_R2-C}
    \trans{(\u_2)} Q \u_2 - \trans{(\u_2)}_0 Q (\u_2)_0 + \sum_{\overline{k}=1}^{\overline{K}} 
    (\overline{h}_{\overline{k}}(\z))^2 
    \overline{\lambda}_{\overline{k}}(\z) = \trans{(\overline{\h}(\z)\otimes \u_2)} \overline{U} ((\overline{\h}(\z))\otimes \u_2),
  \end{equation}	
with $\overline{\lambda}_{\overline{k}}(\z) \in \Sigma_{4-2d_{\overline{k}}}(\z)$ and 
some $\overline{U} \in \SymMat_+^{\overline{K}N}$.

From the assumption that $\trans{C}C$ is invertible in \cref{sec:HiddenConnectivity},
$\overline{\h}(\z) = 0$ if and only if $\h(\z) = 0$.
Since $C \in \Real^{\overline{K} \times K} $ is a matrix, there exists 
$\overline{C} \in  \Real^{(\overline{K} N) \times (KN)}$
such that 
$(C \h(\z))\otimes \u_2 = \overline{C} (\h(\z) \otimes \u_2)$.
By letting $U = \trans{\overline{C}} \ \overline{U} \ \overline{C} \in \SymMat_+^{KN}$, 
\eqref{eq:equality_of_Q_and_R2-C} can be written as
\begin{equation*} 
    \trans{(\u_2)} Q \u_2 - \trans{(\u_2)}_0 Q (\u_2)_0 + \sum_{\overline{k}=1}^{\overline{K}} 
    (\overline{h}_{\overline{k}}(\z))^2 
    \overline{\lambda}_{\overline{k}}(\z) = \trans{(\h(\z)\otimes \u_2)} U (\h(\z)\otimes \u_2).
  \end{equation*}	
  
In addition, 
   $\sum_{\overline{k}=1}^{\overline{K}} 
    (\overline{h}_{\overline{k}}(\z))^2 
    \overline{\lambda}_{\overline{k}}(\z)$ 
can be expressed as 
$\trans{(\h(\z)\otimes \u_2)} V (\h(\z)\otimes \u_2)$
for some  $V \in \SymMat_+^{KN}$.
  Therefore, 
the discussion  in  \cref{sec:hidden}  can be reduced 
to that of \cref{sec:expression-with-h}, and 
we can obtain the exact recovery of the problem with 
$\overline{\h}(\z) = 0$ by solving \eqref{eq:equality_of_Q_and_R2-C}.

\section{Algorithms and numerical results} 
\label{sec:algorithm}

We propose a two-stage algorithm for solving  the SOS relaxation \eqref{eq:sos_relaxation}.
In the first stage, $Q$ is constructed  as described in \cref{lemma:decomposition_of_El} and \cref{thm:rank_recover} with
 the given data $\h(\z)=0$.
In the second stage, we solve \eqref{eq:sos_relaxation}.

We also present numerical results  in \cref{sec:Numerical_results} by implementing the proposed algorithm on \cref{ex:partial_matrix} and test problems with data generated based on
bipartite graphs. 

\subsection{Algorithms}

The  first stage of our proposed algorithm   constructs
 $Q$ in \cref{thm:rank_recover} by   finding $Q^r \ ( r=2,\ldots,N)$ and then
letting $Q = \sum_{r=2}^N Q^r$. Each $Q^r$ has nonzero elements only in ${[Q^r]}_{11}$,  ${[Q^r]}_{rr}$, 
${[Q^r]}_{1 r}$, and ${[Q^r]}_{r 1}$. The resulting $Q$ is an arrowhead matrix.
The second stage is to solve the SOS relaxation to determine  $\Gamma$ in \cref{thm:rank_recover_gen}.

Specifically, in the first stage, for any  $r \in \{2, \dots, N\}$, $Q^r$ is determined as follows: 
\begin{equation} \label{eq:stage1_algorithm} \tag{$\mathscr{S}_1(r)$}
    \begin{array}{rl}
        \mathrm{find} & Q^r \in \SymMat_+^{N},\; U^r \in \SymMat_+^{KN} \\
        \subto
        & \trans{\begin{bmatrix}1 \\ [\u_2]_r\end{bmatrix}} Q^r_{\{1, r\}}\begin{bmatrix}1 \\ [\u_2]_r\end{bmatrix} = \trans{\left(\h(\z) \otimes \u_2\right)} U^r \left(\h(\z) \otimes \u_2\right), \\[2ex]
        & {[Q^r]}_{rr} = 1,  \quad
          {[Q^r]}_{ij} = 0, \quad (i, j) \in  (N \times N), \; i \not\in \{1, r\} \lor j \not\in \{1, r\},
    \end{array}
\end{equation}
where $Q^r_{\{1, r\}}$ denotes a  $2 \times 2$ submatrix that includes the first and $r$th columns and rows of $Q^r$.
  We note that  $\trans{\left(\h(\z) \otimes \u^2\right)} R^r \left(\h(\z) \otimes \u^2\right)$ can be expressed as $\trans{\h(\z)} \widehat{U}^r \h(\z) $
    if  $\widehat{U}^r$ is a $K \times K$  symmetric polynomial matrix.  
For the feasibility of \eqref{eq:stage1_algorithm}, we present the following proposition:
\begin{prop} \label{prop:feasibility_of_stage1}
    For  $r \in \{2, \dots, N\}$, the problem~\eqref{eq:stage1_algorithm} has a feasible solution.
\end{prop}
\begin{proof}
It is easy to see that $\trans{\left(\h(\z) \otimes \u_2\right)} U^r \left(\h(\z) \otimes \u_2\right)$
   in \eqref{eq:stage1_algorithm} is
    a general form of  $\sum_{k = 1}^K h_k(\z)^2 p_k(\z)$,
    where  $p_k(\z) \in \Sigma_{4-2d_k}[\z]$.
   As discussed in \cref{sec:given-data}, 
 there exists $U^r$ such that 
  \[  \trans{\begin{bmatrix}1 \\ \ [\u_2]_r\end{bmatrix}} Q^r_{\{1, r\}}\begin{bmatrix}1 \\ [\u_2]_r\end{bmatrix} =
   \trans{\left(\h(\z) \otimes \u_2\right)} U^r \left(\h(\z) \otimes \u_2\right). 
   \] 
\end{proof}
\noindent
By letting $Q = \sum_{r=2}^N Q^r$, 
we obtain the following problem from \eqref{eq:stage1_algorithm}:
\begin{equation} \label{eq:stage1_algorithm_alter} \tag{$\mathscr{S}_{1,\mathrm{sum}}$}
    \begin{array}{rl}
        \mathrm{find} & Q \in \SymMat_+^{N},\; U \in \SymMat_+^{KN} \\
        \subto
        & \trans{\left(\u_2\right)} Q \u_2 = \trans{\left(\h(\z) \otimes \u_2\right)} U \left(\h(\z) \otimes \u_2\right), \\
        & Q_{ii} = 1, \quad i \in \{2, \dots, N\},  \\
        & Q_{ij} = 0, \quad (i, j) \in  \{2, \dots, N\} \times \{2, \dots, N\},  \; i \neq j,
    \end{array}
\end{equation}
which has a feasible solution by \cref{prop:feasibility_of_stage1}.
For $\sum_{r=2}^N Q^r$,   
\eqref{eq:stage1_algorithm} needs to be solved $(N-1)$ times for each $r$ in $\{2,\dots, N\}$,
while \eqref{eq:stage1_algorithm_alter} needs to be solved once.
Hence, solving \eqref{eq:stage1_algorithm_alter} is  more efficient than \eqref{eq:stage1_algorithm} 
and is employed in our algorithm.

Let $f_\mathrm{sum}(\z) = \trans{\left(\u_2\right)} Q \u_2$.
In the second stage, we solve
the following SDP for a solution that can recover  the matrix.
\begin{equation} \label{eq:stage2_algorithm} \tag{$\mathscr{S}_2(f_\mathrm{sum})$}
    \begin{array}{rl}
        \max\limits_{\rho, \boldsymbol{\lambda}, \Delta_1, \ldots, \Delta_K} & \rho \\
        \subto
        & f_\mathrm{sum}(\z) - \rho + \sum\limits_{k = 1}^K h_k(\z)^2 \lambda_k(\z)
            \in \Sigma_4[\z] \\
		& \lambda_k(\z) = \trans{\u_1} \Delta_k \u_1 \quad (k = 1, \ldots, K) \\
		& \mu I - \Delta_k \succeq O \quad (k = 1, \ldots, K) \\
        & \rho \in \Real, \; \lambda_k(\z) \in \Sigma_{4 - 2d_k}[\z], \; \Delta_k \in \SymMat_+^{m+n} \quad (k = 1, \ldots, K).
    \end{array}
\end{equation}
where $\mu \in \Real$ is a sufficiently large constant.
Compared with \eqref{eq:sos_relaxation},
the problem~\eqref{eq:stage2_algorithm} includes the constraints $\lambda_k(\z) = \trans{\u_2} \Delta_k \u_2$
and $\mu I - \Delta_k \succeq O$ for all $k = 2, \ldots, N$
for the boundedness of the dual variables $\boldsymbol{\lambda}$.
By Slater's condition,
 strong duality holds between \eqref{eq:stage2_algorithm} and its dual.
In the numerical experiments, we use $\mu = 10^6$.

The proposed algorithm is described in \cref{alg:solve_after_making_fsum}. 
\begin{algorithm}[ht]
    \begin{algorithmic}[1]
    \caption{ for solving \eqref{eq:matrix_completion}}
    \label{alg:solve_after_making_fsum}
        \REQUIRE  $\h(\z)$ 
        \ENSURE  A solution $\z^*$ to \eqref{eq:matrix_completion} 
            \STATE
              Solve the problem~\eqref{eq:stage1_algorithm_alter}
                and obtain a solution $\left( Q^*, U^*\right)$.
            \STATE
               $f_\mathrm{sum} \leftarrow  \trans{(\h(\z) \otimes \u_2)} U^* (\h(\z) \otimes \u_2)$.
        \STATE
            Solve the problem~\eqref{eq:stage2_algorithm} using $f_\mathrm{sum}$
            and obtain a solution $(\rho^*, \boldsymbol{\lambda}^*)$.
        \STATE
            Find  the Gram matrix $\Gamma \in \bS^{N}$ such that 
			\begin{equation*}
				\trans{\left(\u_2\right)}\Gamma\u_2
				= f_\mathrm{sum}(\z) - \rho^* + \sum\limits_{k = 1}^K h_k(\z)^2 \lambda^*_k(\z).
			\end{equation*}
		\STATE
			Find a vector $\u_2^* \in \Real^{N}$ in the null space of $\Gamma$.
        \STATE
			$\z^* \leftarrow \frac{1}{(\u_2^*)_1} \trans{\left[(\u_2^*)_2, (\u_2^*)_3,  \dots, (\u_2^*)_{s+1}\right]}$ and return $\z^*$.
    \end{algorithmic}
\end{algorithm}

\noindent
If the rank of $\Gamma$ at  Step~4 is $(N-1)$, the dimension of the null space of $\Gamma$ must be $1$
and
$u_2^*$ is uniquely determined up to a scalar.
At Step~6, the multiplier $\frac{1}{(\u_2^*)_1}$ is used to have $(\u_2^*)_1 = 1$ by the definition of $\u_2$; more precisely, 
the first $s+1$ elements of $\u_2^*/(\u_2^*)_1$ correspond to $[1; \z^*]$.

\subsection{Numerical results} \label{sec:Numerical_results}
We present numerical results by   \cref{alg:solve_after_making_fsum} on \cref{ex:partial_matrix}
in \cref{sec:Example22} and randomly generated problems
 with bipartite graphs in \cref{sec:Random}.
The bipartite graph  $\GC$ is employed as $x_1 = 1$, $x_i \ (i=2,\ldots,n)$, and $y_j \ (j=1,\ldots,m)$ form  two groups and
the partially given data for $x_iy_j$ can be regarded as an edge between the two groups.

The purpose of the experiments is to  verify that \cref{alg:solve_after_making_fsum} can recover the exact solution of \eqref{eq:matrix_completion}
under the assumption of \cref{thm:rank_recover} and
the connectivity of the bipartite graph $\GC$. 
In the discussion of numerical results by \cref{alg:solve_after_making_fsum},
 the error is computed by  
\begin{align} \label{eq:def_error}
	\text{Er$_{\z^*}$} := \frac{\left\| \z^* - \z_0 \right\|_2}{\left\|\z_0\right\|_2}.
\end{align}
All computations were performed on 
Intel Core i9-12900K (3.2GHz) with Julia~1.9.2 and Mosek~10.0.2.

\subsubsection{Illustrative example} 
\label{sec:Example22}

For \cref{ex:partial_matrix},
 the first stage of the algorithm solves~\eqref{eq:stage1_algorithm_alter}
with  $\h(\z) = [y_1 - 7; \; y_2 - 3; \; x_2 y_1 + 35; \; x_2 y_3 - 10; \; x_3 y_2 - 9]$.
The resulting  polynomial $f_\mathrm{sum}$ is used in  the second-stage problem~\eqref{eq:stage2_algorithm}:
\begin{equation*}
    \begin{array}{rl}
        \max\limits_{\rho, \boldsymbol{\lambda},\Delta_1,\Delta_2} & \rho  \\
        \quad \subto
        &  f_\mathrm{sum} - \rho + (y_1 - 7)^2 \lambda_1 + (y_2 - 3)^2 \lambda_2  \\
        &  \quad + (x_2 y_1 + 35)^2 \lambda_3 + (x_2 y_3 - 10)^2 \lambda_4 + (x_3 y_2 - 9)^2 \lambda_5 \; \in \Sigma_4[\z] \\[1ex]
        & \lambda_k(\z) = \trans{\u_1} \Delta_k \u_1 \quad (k = 1, 2) \\
        & \mu I - \Delta_k \succeq O \quad (k = 1,2) \\
        & \rho \in \Real, \; \Delta_k \in \SymMat_+^{n+m} (k = 1, 2), \; \lambda_3, \lambda_4, \lambda_5 \in [0, \mu].
    \end{array}
\end{equation*}
The matrices $Q$ and $U$ in \eqref{eq:U_def} can be obtained after \eqref{eq:stage1_algorithm_alter} is solved.
Their size  is so large  that  the matrices cannot be displayed here, instead, shown in GitHub (\url{https://github.com/godazm/exact-matrix-completion}).

\cref{table:check_illustrative_example} displays
the total time for the two stages, the optimal value of \eqref{eq:stage2_algorithm},
the recovered solution $\z^*$ of \cref{alg:solve_after_making_fsum}, and its error $\text{Er$_{\z^*}$}$.
We can see 
 that the total time for the first stage is longer than that of the second stage,
although \eqref{eq:stage1_algorithm_alter} was implemented instead of \eqref{eq:stage1_algorithm}.
We observe that \eqref{eq:stage2_algorithm} can recover an exact solution of the rank-1 matrix for~\cref{ex:partial_matrix}.

\begin{table}[htb]
	\caption{Average total time in seconds and relative errors.}
	\label{table:check_illustrative_example}
	\centering
	\begin{tabular}{ccccc}
		\multicolumn{2}{c}{Total time} & Optimal value & Recovered solution $\z^*$ & $\text{Er$_{\z^*}$}$ \\
		\eqref{eq:stage1_algorithm_alter} & \eqref{eq:stage2_algorithm} & \eqref{eq:stage2_algorithm} & &   \\ \hline
		$7.62$ & $1.03$ & $1.59 \times 10^{-6}$ & $[1.00; -5.00; 3.00; 7.00; 3.00; -2.00]$ & $9.91 \times 10^{-7}$ \\
	\end{tabular}
\end{table}


\subsubsection{Tests on random data} \label{sec:Random}

We performed numerical  experiments on two types of  random data: data in a chain format and
data with hidden connectivity.  
Each experiment involved 100 randomly generated instances.
The instances were generated by  \cref{alg:random_bipartite_graph} for connected bipartite graphs.
With $\mathrm{rand}(\cdot)$ in \cref{alg:random_bipartite_graph},
each element of $(\cdot)$ is generated 
with the same probability. 
Ten samples of  $\z_0$ were generated simultaneously
where each element is chosen by $(\z_0)_i \in \mathrm{rand}([-5, 5])$ for $i \in \{1, \ldots, s\}$.
\begin{algorithm}[ht]
    \begin{algorithmic}[1]
    \caption{Generating random bipartite graphs}
    \label{alg:random_bipartite_graph}
        \REQUIRE The number of vertices $v$
        \ENSURE A bipartite graph $(\{1,\dots, n\}, \{1, \dots, m\}, \Gamma)$
		\STATE $n \leftarrow \mathrm{rand}(\{1, \dots, v-1\})$,  $m \leftarrow v - n$
		\STATE $L_\mathrm{remain} \leftarrow \{1, \dots, n\}$, $R_\mathrm{remain} \leftarrow \{1,\dots, m\}$
		\STATE $l_\mathrm{first} \leftarrow \mathrm{rand}\left(L_\mathrm{remain}\right)$
					and then $L_\mathrm{remain} \leftarrow L_\mathrm{remain} \setminus \{l_\mathrm{first}\}$
		\STATE $r_\mathrm{first} \leftarrow  \mathrm{rand}\left(R_\mathrm{remain}\right)$
					and then $R_\mathrm{remain} \leftarrow R_\mathrm{remain} \setminus \{r_\mathrm{first}\}$
		\STATE $\Gamma \leftarrow \left\{(l_\mathrm{first}, r_\mathrm{first})\right\}$
		\WHILE{$|L_\mathrm{remain}| + |R_\mathrm{remain}| \geq 1$}
			\STATE Select {\em Left}  or {\em Right} randomly
			\IF{{\em Left} and $L_\mathrm{remain} \neq \emptyset$}
			 	\STATE $r \leftarrow  \mathrm{rand}\left(\{1, \dots, m\} \setminus R_\mathrm{remain}\right)$
			 	\STATE $l \leftarrow \mathrm{rand}\left(L_\mathrm{remain}\right)$
			 		and then $L_\mathrm{remain} \leftarrow L_\mathrm{remain} \setminus \{l\}$
			 	\STATE $\Omega \leftarrow \Omega \cup \{(l, r)\}$
			\ELSIF{{\em Right}  and  $R_\mathrm{remain} \neq \emptyset$}
			 	\STATE $l \leftarrow \mathrm{rand}\left(\{1, \dots, n\} \setminus L_\mathrm{remain}\right)$
			 	\STATE $r \leftarrow \mathrm{rand}\left(R_\mathrm{remain}\right)$
			 		and then $R_\mathrm{remain} \leftarrow R_\mathrm{remain} \setminus \{r\}$
			 	\STATE $\Omega \leftarrow \Omega \cup \{(l, r)\}$
			\ENDIF
        \ENDWHILE
    \end{algorithmic}
\end{algorithm}

\begin{table}[ht]
	\caption{Bipartite graphs.}
	\label{table:graph_property}
	\centering
	\begin{tabular}{c|cccccccccc}
		& $\GC_1$ & $\GC_2$ & $\GC_3$ & $\GC_4$ & $\GC_5$ & $\GC_6$ & $\GC_7$ & $\GC_8$ & $\GC_9$ & $\GC_{10}$ \\ \hline
		$n$ & $9$ & $7$ & $6$ & $2$ & $6$ & $2$ & $1$ & $8$ & $1$ & $6$ \\
		$m$ & $1$ & $3$ & $4$ & $8$ & $4$ & $8$ & $9$ & $2$ & $9$ & $4$ \\
		max. degree & $9$ & $5$ & $3$ & $5$ & $4$ & $5$ & $9$ & $6$ & $9$ & $4$
	\end{tabular}
\end{table}

We generated $10$ bipartite graphs $\GC_1,\ldots,\GC_{10}$ associated with
$\X$ in \eqref{eq:matrix_completion} with $10$ vertices fixed.
For the connectivity,
each graph $\GC_p$ always included $9$ edges for $p = 1, \ldots, 10$. 
The graph property of $\GC_p$ is summarized on \cref{table:graph_property}.
The lengths of $\x$ and $\y$ are shown in the rows of $n$ and $m$, respectively,
and the maximum degree among all vertices in $\GC_p$ is also shown in the last row.
For each $\z_0$ and each bipartite graph $\GC_p$,
 we generated $100$ instances of \eqref{eq:matrix_completion} as follows:
 (i)   The initial $n-1$ elements of $\z_0 \in \Real^{n+m-1}$ are regarded as the true values $(\x_0)_2, \dots, (\x_0)_{n}$,
		while the last $m$ elements are considered as $\y_0$.  
 (ii)  For every edge $(i, j) \in \mathbb{N}^2$ of $\GC_p$, an equality constraint $h_k$ is constructed to determine $(X_0)_{ij}$. 
For example, when $n = 7$ and $m = 3$,
the edge $(3, 2)$ corresponds to a constraint of form $h_k(\z) = x_3 y_2 - (\x_0)_3 (\y_0)_2 = 0$,
where $\x_0 = \trans{[1, (\z_0)_1, \ldots, (\z_0)_6]}$ and $\y_0 = \trans{[(\z_0)_7, (\z_0)_8, (\z_0)_9]}$.

\begin{table}[ht]
	\caption{Average total time in seconds and relative errors.}
	\label{table:comparinggraph}
	\centering
	\begin{tabular}{c|ccc}
		& \multicolumn{2}{|c}{Total time} & Mean errors \\
		& \eqref{eq:stage1_algorithm_alter} & \eqref{eq:stage2_algorithm} & \\ \hline
		$\GC_1$    & $68.9$ & $0.325$ & $6.21 \times 10^{-8}$ \\
		$\GC_2$    & $119 $ & $0.338$ & $2.81 \times 10^{-7}$ \\
		$\GC_3$    & $138 $ & $0.338$ & $1.93 \times 10^{-6}$ \\
		$\GC_4$    & $87.7$ & $0.349$ & $3.58 \times 10^{-7}$ \\
		$\GC_5$    & $148 $ & $0.340$ & $7.51 \times 10^{-6}$ \\
		$\GC_6$    & $88.2$ & $0.354$ & $1.40 \times 10^{-6}$ \\
		$\GC_7$    & $42.7$ & $0.354$ & $4.65 \times 10^{-8}$ \\
		$\GC_8$    & $105 $ & $0.343$ & $7.02 \times 10^{-7}$ \\
		$\GC_9$    & $45.1$ & $0.365$ & $3.46 \times 10^{-8}$ \\
		$\GC_{10}$ & $138 $ & $0.348$ & $3.28 \times 10^{-7}$ \\
	\end{tabular}
\end{table}

\cref{table:comparinggraph} shows the results on random data in a chain structure. 
Each row displays the mean value of times in seconds for the first and second stages,
and the mean of errors~\eqref{eq:def_error} from $10$ samples.
We observe that the first stage spent much computational time as
\eqref{eq:stage1_algorithm_alter} involves a large matrix of size $495 \times 495$.
Here $495$ is computed from $9 \times 55$ with the size of $\h(\z)$ and  $\u_2$, which is $K = 9$ and  $N = \begin{pmatrix} 11 \\ 2 \end{pmatrix} = 55$, respectively.
We  also see that the total time is shorter for the instances with larger differences between  $n$ and $m$  than the instances with small differences. 
The errors shown in the last column of \cref{table:comparinggraph} are small, thus
we see that the exact solution of \eqref{eq:matrix_completion_first} can be
recovered for all instances.

To test the second type of data for \eqref{eq:matrix_completionH}, test instances were generated with a matrix $C$
as discussed in \cref{sec:HiddenConnectivity}.
The $i$th row of $C$ scales each component of $\h(\z)$,
i.e., if the first row of $C$ is $[1 \; 0 \; {-3} \; 0 \; 0 \; \dots \; 0]$,
the first constraint of \eqref{eq:matrix_completionH} must be $\left(h_1(\z) - 3 h_3(\z)\right)^2 \leq 0$.
In our experiments,
we generated $10$ invertible matrices $C_1, \ldots, C_{10} \in \Real^{9 \times 9}$ for $C$.
Each element of them was randomly chosen from $[-5, 5]$,
 and then  each row of $C_k$ was normalized using the Euclidean norm for  $k=1, \ldots, 10$.
For each matrix $C_k$,
we tested the algorithm for all the combination of graphs $\GC_1,\ldots,\GC_{10}$
and $10$ samples of  $\z_0$, and thus $100$ instances were used.

 \cref{fig:c_error,fig:c_time}
show the performance of the algorithm for each matrix $C_k$.
Note that \eqref{eq:matrix_completion} can be regarded as a special case of \eqref{eq:matrix_completionH}, in other words, \eqref{eq:matrix_completionH} with $C = I$.
From 
the error $\text{Er$_{\z^*}$}$ from $100$ instances in \cref{fig:c_error},
the proposed algorithm can extract the solution with accuracy even for the problems with 
the hidden connectivity. We also notice that the mean and maximum
 errors for the instances with $C= C_1, C_2, \ldots, C_{10}$ 
 are larger than that for the instances with $ C=I$.
The computing time for the first and second stages is shown in \cref{fig:c_time}.
The time unit in this figure is second.
We see 
that 
 the total time for the instances with $C= C_1, C_2, \ldots, C_{10}$ 
 is much longer than that for the instances with $ C=I$,  
 since $\overline{\h}(\z) = C \h(\z)$ involves more nonzero coefficients than 
 $\h(\z)$.

\begin{figure}[ht]
	\centering
	\begin{minipage}{0.47\linewidth}
		\centering
		\includegraphics[keepaspectratio, width=0.99\linewidth]{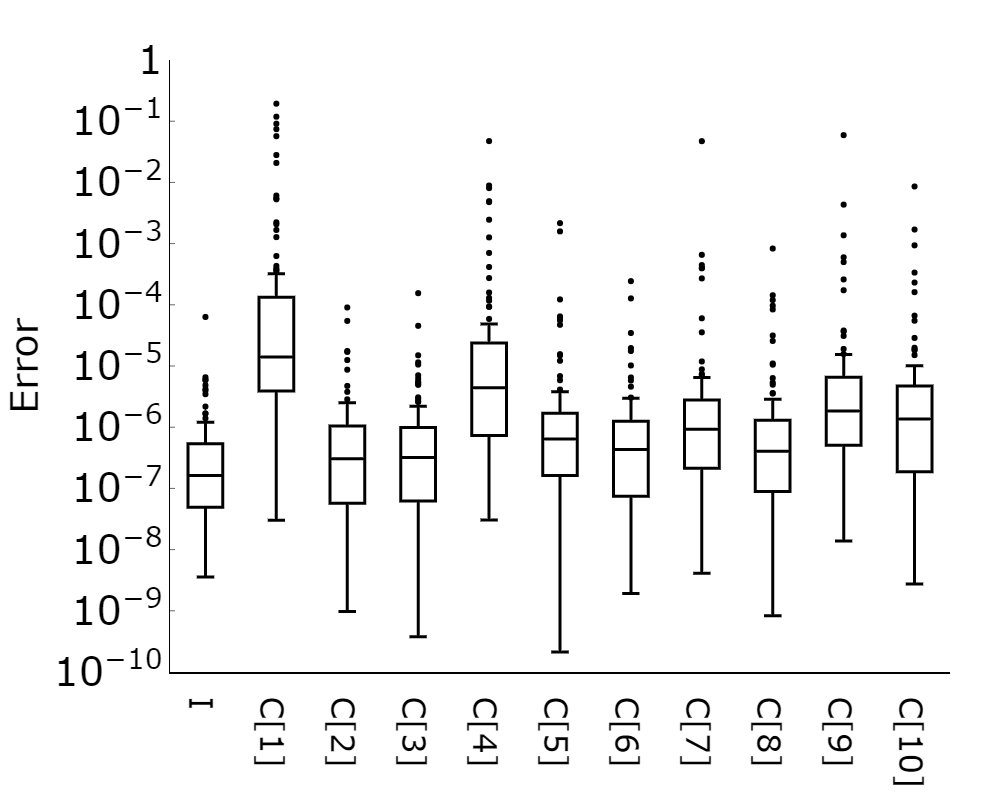}
	\end{minipage}
	\caption{Relative errors for the instances with hidden connectivity,  $I, C_1, \ldots, C_{10}$. Errors are shown in the log-scale.}
	\label{fig:c_error}
	\centering
	\begin{minipage}{0.47\linewidth}
		\centering
		\includegraphics[keepaspectratio, width=0.99\linewidth]{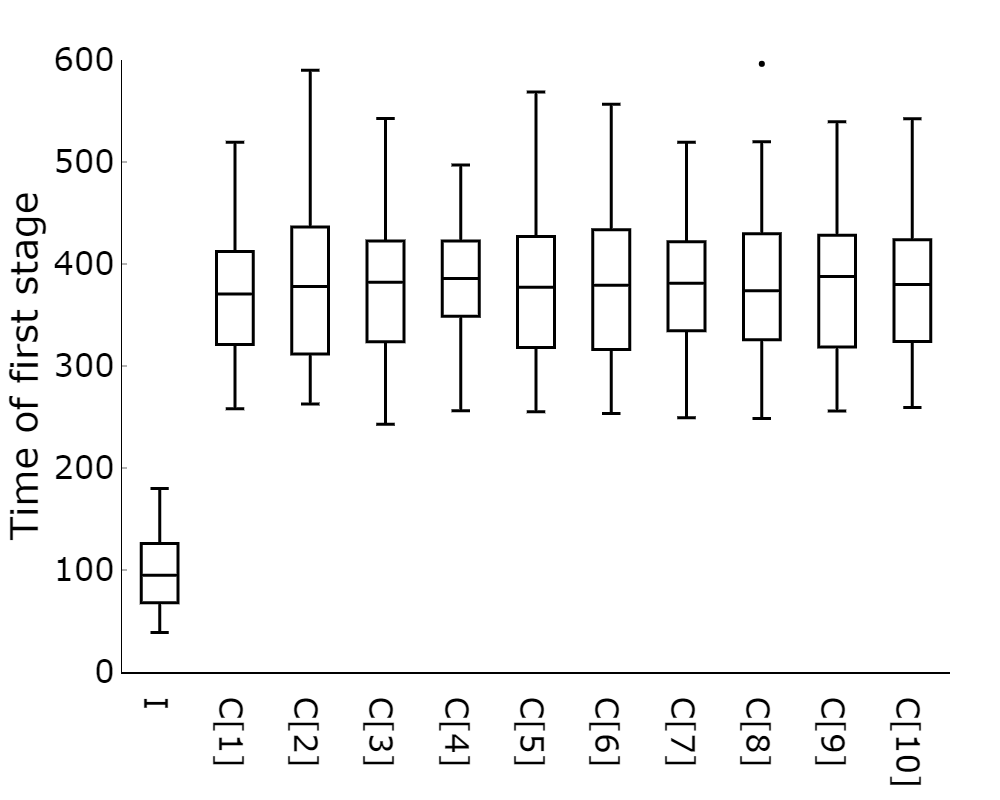}
		\subcaption{First stage}
	\end{minipage}
	\hspace{1em}
	\begin{minipage}{0.47\linewidth}
		\centering
		\includegraphics[keepaspectratio, width=0.99\linewidth]{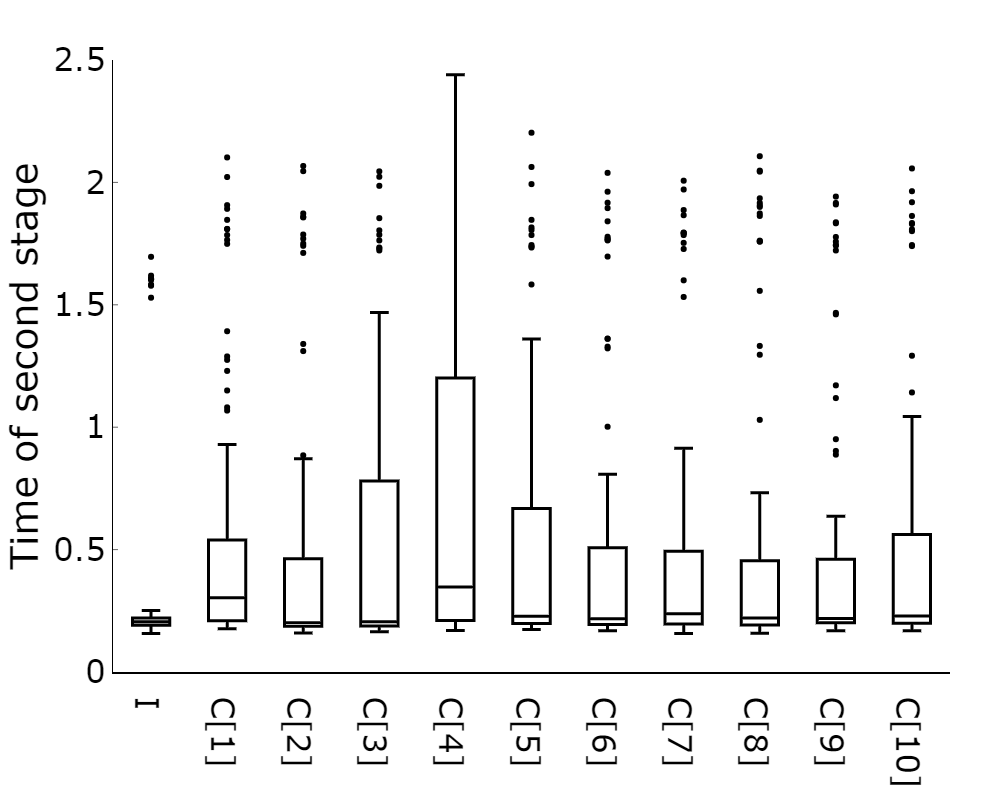}
		\subcaption{Second stage}
	\end{minipage}
	\caption{Total time in seconds for the instances with hidden connectivity, $I, C_1, \ldots, C_{10}$.} 
	\label{fig:c_time}
\end{figure}

\section{Concluding remarks} \label{sec:conclusion}

We have studied a minimization problem arising from the matrix completion under the assumption that
the known values of  the matrix elements form a chain format or
possess the hidden connectivity. 
For the minimization problem, 
we have proved that the exact solution to the problem can be obtained by the SOS relaxation or the SDP relaxation
by formulating the problem with the constraints as polynomials of degree up to 4 and incorporating an objective function with the arrowhead sparsity pattern. 

Exploiting the sparsity of a given problem has been an essential subject for improving the computational efficiency of  SDP relaxations \cite{fukuda2001exploiting,kim2011exploiting,KIM2009,nakata2003exploiting}.
It has been commonly employed to reduce the size of the variable matrix so that the resulting problem 
can be solved fast.
The bipartite structure, on the other hand, has been used to show the exactness of the SDP relaxation for some class of quadratically constrained
quadratic programs \cite{Azuma2021, Azuma2022}.

The sparsity discussed in this paper falls into two categories:
one is for the objective function and the other for the constraints.  
In the case of  the proposed problem  \eqref{eq:matrix_completion}, the sparsity in the objective function is not obtained from the given problem.
Rather, it is enforced onto the objective  function through
the matrix $Q$ in \eqref{eq:sos_relaxation} with the arrowhead sparsity pattern and  the information from the given
data.
The constraint sparsity, characterized by the chain format, additionally preserves the rank 
of $\Gamma$ in \cref{thm:rank_recover_gen} 
as $N-1$.
In this context, both forms of sparsity have been utilized to show the exactness of the solution obtained from the SOS relaxation.
It will be interesting to explore the applicability of our approach to other problems where the objective function can be chosen.

\vspace{0.5cm}

\noindent



\end{document}